\documentclass[letterpaper, 10 pt, conference]{ieeeconf}  

\IEEEoverridecommandlockouts                              
\usepackage[letterpaper,
            left=.75in,
            right=.75in,
            top=1in,
            bottom=.75in]{geometry}
\usepackage{subcaption}
\usepackage{cite}
\usepackage{amsmath,amssymb,amsfonts}
\usepackage{algpseudocode}
\usepackage{algorithm}
\usepackage{graphicx}
\usepackage{textcomp}
\usepackage{hhline}
\usepackage[abs]{overpic}
\usepackage{mathrsfs}
\usepackage{xcolor}
\usepackage{hyperref}
\usepackage{lettrine}
\definecolor{violet}{rgb}{.5,0,.5}
\definecolor{orange}{rgb}{1,.65,0}

\newcommand{\tr}{\text{tr}}

\newcommand{\calE}{\mathcal{E}}

\newcommand{\bbR}{\mathbb{R}}
\newcommand{\bbS}{\mathbb{S}}

\newtheorem{theorem}{Theorem}[section]

\newtheorem{remark}[theorem]{Remark}
\newtheorem{assumption}[theorem]{Assumption}
\newtheorem{prop}[theorem]{Proposition}
\newtheorem{definition}[theorem]{Definition}

\def\BibTeX{{\rm B\kern-.05em{\sc i\kern-.025em b}\kern-.08em
    T\kern-.1667em\lower.7ex\hbox{E}\kern-.125emX}}
\begin{document}
\title{Linear Output Feedback and Guidance Synthesis:\\ Framework for Jointly Design Guidance and Robust Feedback}
\title{Joint Observer-Based Output-Feedback\\Robust Guidance and Control Funnel Synthesis}
\author{Cheng Chang,  Shiva Shakeri, \IEEEmembership{Members, IEEE} and Mehran Mesbahi \IEEEmembership{Fellow, IEEE}
\thanks{The research of the authors have been supported by Air Force Office of Scientific Research grant FA9550-26-1-B206.}
\thanks{The authors are with the William E. Boeing Department of Aeronautics and Astronautics, University of Washington. Emails: \{chang53+sshakeri+mesbahi\}@uw.edu}}

\maketitle

\begin{abstract}
A joint guidance and robust observer-based output-feedback synthesis framework is developed for discrete-time nonlinear systems subject to bounded process disturbances and measurement noise. The method co-designs a reference trajectory, time-varying Luenberger observer and observer-based linear output-feedback gains, and coupled ellipsoidal funnels that certify conditioned-invariant bounds on tracking and estimation errors. Nonlinear residuals are bounded by incremental quadratic constraints (QCs) parameterized by local Lipschitz constants, estimated via sampling approaches. The coupled stability conditions are formulated as linear matrix inequality (LMI) surrogates and solved using sequential convex programming with proximal regularization (prox-SCP) via semidefinite programming (SDP). Simulations of a unicycle model under persistent sensory noise and system disturbances demonstrate robust tracking and funnel-certified constraint satisfaction with respect to the jointly designed reference trajectory.
\end{abstract}

\section{Introduction}
\label{sec:introduction}

\lettrine{P}lanning a reference trajectory and control synthesis are often treated as separate design problems.
For safety-critical systems operating under uncertainty, this separation is
untenable: the trajectory must leave sufficient margin for the controller to ``absorb''
disturbances, and the controller must be designed knowing what margin it has
\cite{mayne_mpc,majumdar_tedrake}.
Robust trajectory synthesis addresses this design dependency by co-designing the reference
with an invariant set, a funnel or tube, that bounds the worst-case deviations
of the true state under bounded disturbances
\cite{majumdar_tedrake,mayne_tube,rakovic_mrpi}.
If this set remains inside the feasible region along the trajectory, then constraint satisfaction is certified for all admissible disturbance realizations before execution \cite{mayne_mpc,taylor}, a guarantee of direct relevance to
powered descent vehicles, autonomous ground robots, and agile aerial
platforms \cite{acikmese_ploen,herbert_fastrack,singh_contraction}.

Recent work has extended this design paradigm to nonlinear discrete-time systems,
co-designing trajectories and ellipsoidal funnels via successive
convexification and linear matrix inequalities (LMIs)~\cite{twewan1,taylor,leeman,manchester_kuindersma}.
However, these works often assume that the full state information is
directly available for feedback control.
In practice, sensors measure a noisy proper subset of the state,
and the estimation error enters the closed-loop deviation dynamics through
the observer \cite{mayne_of,xu_obs}.
A funnel synthesized under the perfect state knowledge assumption
does not, in general,
remain invariant when the controller operates on estimated states
\cite{atassi_khalil}.

Separation principle that for linear systems allows a stable observer
and a stable controller to be combined without compromising the stability
of the closed-loop system, does not extend to nonlinear systems operating under hard
constraints~\cite{atassi_khalil}: in such a setting, estimation and tracking errors are coupled
through the closed-loop dynamics, and any valid invariance certificate must
 explicitly account for this coupling.
 In this direction, Mayne {\em et al.}~\cite{mayne_of} has shown 
 that for constrained linear
discrete-time systems a Luenberger observer \cite{luenberger} and a
tube-based controller can be combined by maintaining two separate
invariant sets.
For nonlinear systems, the work of Xu {\em et al.}~\cite{xu_obs} has proposed
co-designing observer and controller gains via LMIs under incremental quadratic constraints
\cite{bechet,QC}, but for stabilization at a fixed equilibrium, without
separate geometric bounds on tracking and estimation errors.
Contraction-based methods \cite{lohmiller_slotine,ccm,singh_contraction}, on the other hand, 
certify robust tracking along arbitrary trajectories but assume full state
feedback.

To this end, in this work we propose a framework that jointly synthesizes a reference trajectory, 
time-varying observer and observer-based output-feedback gains, 
and coupled ellipsoidal funnels for discrete-time nonlinear systems with bounded process and measurement disturbances.
The two funnels are separately certified but jointly constrained: one bounds the 
tracking error and the other bounds the estimation error, and the LMI conditions
derived via the S-procedure and the Schur complement explicitly account for
their interaction.
Nonlinear dynamics are handled via incremental quadratic constraints
\cite{bechet,QC} with local Lipschitz constants estimated by sampling on
the funnel boundaries \cite{taylor,taylor_phd,twewan1}, and the
nonconvex joint synthesis problem is solved via prox-SCP  \cite{scvx_main,Moreau_regularization}. The contributions and comparison against existing methods are summarized in Table \ref{table: comparison}.

\begin{table}[t]
\caption{Comparison between existing methods and the current work}
\label{table: comparison}
\begin{center}
\begin{tabular}{ccccc}

Methods & Nominal & Feedback & Funnel types &  
\\
&traj opt&information&& \\
\hhline{=====}
Our work & Single opt & Observer-  & State/ Observer &  \\ 
&&feedback&&
\\\hline
Original & None & Full-state & State/ Control & \\
Funnel\cite{taylor}&&&& \\\hline

Joint full-state \cite{twewan1} & Bilevel opt  & Full-state & State/ Control &  \\
feedback &  &  &  &\\ \hline
Observer& None  & Observer- & Joint state +   &   \\
feedback \cite{xu_obs}&&feedback&observer*&

\\\hhline{=====}
\multicolumn{5}{p{242pt}}{*The joint state+observer funnel in \cite{xu_obs} uses single funnel to evaluate the combined effects from the state and estimation errors }
\end{tabular}
\end{center}
\end{table}

The outline of this paper is as follows. In \S\ref{sec:prelim}, we provide an introduction to the structured nonlinear system and the conditions for the invariance and stability of incremental systems, that will be used in \S\ref{sec:problem} for formulating the synthesis problem. In \S\ref{sec:joint-funnel-synthesis}, the nonconvex synthesis problem is reformulated 
as local convex subproblems that are solved iteratively by the prox-SCP algorithm presented in \S\ref{sec:algorithm}. Numerical simulations are provided in \S\ref{sec:numerical-simulations} prior to concluding the discussion on the limitations and advantages of the proposed framework in \S\ref{sec:conclusion}.
\section{Preliminaries}
\label{sec:prelim}
In this section, we provide relevant technical background and constructs that are used for our subsequent discussion.

\subsection{Open-Loop System Dynamics}
\label{subsec: dyanmics}
Consider the continuous-time nonlinear system with linear measurement output captured by the continuous system mapping $f: \bbR^{n_x}\times\bbR^{n_u} \times \bbR^{n_w}\to \bbR^{n_x}$, subject to state  constraints defined by $h:\bbR ^{n_x} \to \bbR$ 
\begin{equation}\label{eq:ct-dynamics}
\begin{aligned}
    \dot{x} = f(x, u,w), \; y = Cx + Dv,  \; h(x)\leq 0,
\end{aligned}
\end{equation}
where $x \in \bbR^{n_x}$, $y \in \bbR^{n_y}$, $u \in \bbR^{n_u}$, $w \in \bbR^{n_{w}}$, and $v \in \bbR^{n_v}$ are the state, output, control input, process noise, and measurement noise terms, respectively; $C \in \bbR^{n_y \times n_x}$ is the measurement channel and $D \in \bbR^{n_y \times n_v}$ designates the measurement noise channel.
In order to quantitatively analyze the effects of nonlinearities in the system, we restrict our attention to systems that can be decomposed as follows:

\begin{assumption}[Structured and differentiable system]\label{as:str nonliner}
The dynamics in \eqref{eq:ct-dynamics} can be separated into linear and structured nonlinear form by the nonlinearity selecting channels $E \in \bbR^{n_x \times n_p},\quad C_q \in \bbR^{n_q \times n_x},\quad D_q \in \bbR^{n_q\times n_u}$, and $G_q \in \bbR^{n_q \times n_w}$ as follows:
\begin{equation}
\begin{aligned}
    \dot{x} &= A_{\tiny\mbox{lin}} x + B_{\tiny{\mbox{lin}}}u + G_{\tiny\mbox{lin}}w + Ep ,\\
    p &= \phi(q), \quad q = C_qx + D_qu + G_q w,
\end{aligned}
\end{equation}
where $p\in \bbR^{n_p},\quad q \in \bbR^{n_q}$ can be regarded as the structured output and input from the linear process, respectively with the continuous and differentiable nonlinear process mapping $\phi:\bbR^{n_q}\to \bbR^{n_p}$, and $A_{\tiny\mbox{lin}} \in \bbR^{n_x \times n_x} ,\quad B_{\tiny{\mbox{lin}}} \in \bbR^{n_x \times n_u} ,\quad G_{\tiny\mbox{lin}} \in \bbR^{n_x \times n_w}$ represent the system channels for the linear process. For more on the system decomposition, see \cite{taewan} for detailed examples.
\end{assumption}

By the differentiability of $\phi$ made under Assumption \ref{as:str nonliner}, $\phi$ is locally Lipschitz over closed sets within its domain \cite{rudin}.

The finite dimensional representation of the trajectory is
obtained via discretization of \eqref{eq:ct-dynamics} with a time step $dt$ leading to,
\begin{equation}\label{eq:dt-dynamics}
\begin{aligned}
    x_{k+1} = f_d(x_k, u_k, w_k), \, y_k = Cx_k + Dv_k, \, h(x_k) \leq 0,
\end{aligned}
\end{equation}
where $k = 0, \ldots, T{-}1$ is the discrete time index, $w_k \in \bbR^{n_w}$ denotes the discrete-time disturbance induced by discretization (e.g., integrated process noise over $[t_k, t_{k+1}]$), and $v_k \in  \bbR^{n_v}$ represents the sensor noise measured at time $k$. In general, more than one set of realizations of channel matrices $E,\;C_q,\;D_q$, and $G_q$ exists that are not necessary minimal.
\begin{assumption}[Exact channel preservation]
Assume there exist nonlinear channel selection matrices $C_q,\;D_q$, and $G_q$ that remain unchanged between the discrete and continuous representation of the system, and $E_k  = h E, \; h>0$ is the rescaled discrete nonlinear input matrix. 
\end{assumption}
In order to simplify the notation, we denote the time set as $\mathcal{N}_0^t = \{k: k =0,\ldots,t\}$.
For robust feedback synthesis through quadratic constraints (QCs) on the Lur'e like systems over the bounded uncertainty sets, we further assume the following.
\begin{assumption}[Normalized bounded disturbances]\label{as:noise}
For all $k$, $w_k^\top  w_k \leq 1$ and $v_k^\top v_k \leq 1$. 
\end{assumption}
Note that the choice of normalization bounds can be arbitrary positive scalars, as long as it is properly scaled with respect to the funnel to be introduced subsequently in \S\ref{subsec:funnels} \cite{taylor,taewan,twewan1}.
\begin{assumption}[Differentiable discrete system]\label{as:diff}
For all $k$, $f_d$ is differentiable with respect to $x_k,\; u_k$, and $w_k$. Similarly, the functional representation of state constraints $h$ is assumed to be differentiable with respect to $x_k$.
\end{assumption}
Since the reference trajectory is designed over the deterministic system, we consider the following undisturbed form of \eqref{eq:dt-dynamics}.
Let $(\bar{x}_k, \bar{u}_k)_{k=0}^{T}$ be a reference system trajectory satisfying the noise-free dynamics,
\begin{equation}\label{eq:reference}
\begin{aligned}
        \bar{x}_{k+1} &= f_d(\bar{x}_k, \bar{u}_k, 0), \quad \bar{x}_0 = x_0, \quad \bar{x}_T = x_T, \\ 
        h(\bar{x}_k) & < 0,
\end{aligned}
\end{equation}
and define the time-varying Jacobians evaluated at $(\bar{x}_k, \bar{u}_k, 0)$ by $A_k = \frac{\partial f_d}{\partial x}|_{\bar{x}_k, \bar{u}_k, 0}$, $B_k = \frac{\partial f_d}{\partial u}|_{\bar{x}_k, \bar{u}_k, 0}$, and $G_k = \frac{\partial f_d}{\partial w}|_{\bar{x}_k, \bar{u}_k, 0}$.
Note the strict feasibility requirement for the reference trajectory 
as it has to be ``certified'' with \emph{nonzero margins}; this
will be further elaborated upon in \S\ref{subsec:funnels}.

\subsection{Output-Feedback Incremental Dynamics}
\label{subsec:closed-loop}
Consider the \emph{incremental} output feedback law augmented with the feedforward reference term,
\begin{equation}\label{eq:feedback}
    u_k = \bar{u}_k + K_k(\hat{x}_k - \bar{x}_k),
\end{equation}
with time-varying gain $K_k \in \bbR^{n_u \times n_x}$, where $\hat{x}_k$ is the state estimate provided by the \emph{Luenberger}-type observer,
\begin{equation}\label{eq:observer}
    \hat{x}_{k+1} = f_d(\hat{x}_k, u_k, 0) + L_k(y_k - C\hat{x}_k),
\end{equation}
with gain $L_k \in \bbR^{n_x \times n_y}$. The observer uses a noise-free prediction model and evolves using the innovation $L_k(y_k - C\hat{x}_k)$.

Next, define the state deviation $\eta_k = x_k - \bar{x}_k$, estimation error $e_k = x_k - \hat{x}_k$, and the discretized nonlinearity $\phi_k$ at time $k$. The structured ``decompositional'' nonlinearity in the system dynamics yields an exact incremental representation, with the nonlinear residual terms captured through $\phi_k$ \cite{taewan,xu_obs}. Using $\hat{x}_k - \bar{x}_k = \eta_k - e_k$, we obtain,
\begin{equation} \label{eq:deviation sytem}
\begin{aligned}
    \eta_{k+1}  &= (A_k+B_kK_k)\eta_k - B_kK_k e_k + G_kw_k + E_k (p_k - \bar{p_k}) , \\
    e_{k+1}  &= (A
    _k-L_kC)e_k+ G_kw_k - L_kDv_k + E_k(p_k - \hat{p}_k), \\
    q_k &= C_qx_k + D_qu_k+G_qw_k, \\
    \bar{q}_k & = C_q \bar{x}_k + D_q \bar{u}_k,\;\hat{q}_k = C_q \hat{x}_k + D_q u_k,
\end{aligned}
\end{equation}
with $p_k = \phi_k(q_k)$, $\bar{p}_k = \phi_k(\bar{q}_k)$, and $\hat{p}_k = \phi_k(\hat{q}_k)$ represent the output of the system nonlinearities from the reference system \eqref{eq:reference} and the observer \eqref{eq:observer}. To this end, define $\delta p_k = p_k - \bar{p}_k,\quad \delta q_k = q_k - \bar{q}_k$, $\Delta p_k = p_k - \hat{p}_k $, and $\Delta q_k = q_k - \hat{q}_k$ in the context of the Lipschitz condition to be discussed. Note that the observer term is coupled with the process uncertainty through $G_kw_k$. 

From Assumption \eqref{as:str nonliner}, $\phi_k$ is locally Lipschitz and satisfies, 
\begin{equation}\label{eq:lips}
    \|\delta p_k \|_2^2 \leq \gamma^2_k \| \delta q_k \|_2^2, \quad\|\Delta p_k \|_2^2 \leq \gamma^2_k \| \Delta q_k \|_2^2, 
\end{equation}
over any closed set for $(x,u,w)$ \cite{taylor}, and this will be subsequently used as the matrix multiplier certification for the robust feedback in \S\ref{subsec: quadratic stability}.
\begin{remark}
    Through the structured nonlinearity construction on the incremental system, one can consider the stability of the ``linear'' part \eqref{eq:deviation sytem} subject to the nonlinear feedback and uncertainty effects, as with the standard Lur'e system~\cite{QC}.
\end{remark}


\subsection{Ellipsoidal Funnels}
To efficiently represent the deviation set, we consider bounds of the deviation to be ellipsoidal funnels (tubes).
\label{subsec:funnels}

\begin{definition}[Ellipsoidal funnel]\label{def:funnel}
Given a reference trajectory $(\bar{x}_k)_{k=0}^T$ and a sequence of positive definite matrices $(Q_k)_{k=0}^T$, $Q_k \in \bbS^{n_x}_{++}$, the \emph{state funnel} is the collection of ellipsoidal sets,
\begin{equation}\label{eq:funnel-eta}
    \calE_k^\eta = \big\{ \eta \in \bbR^{n_x} : \eta^\top Q_k^{-1}\eta \leq 1 \big\}, \quad \forall k \in \mathcal{N}^T_0.
\end{equation}
Similarly, given $(P_k)_{k=0}^T$, $P_k \in \bbS^{n_x}_{++}$, the \emph{observer funnel} is the collection,
\begin{equation}\label{eq:funnel-e}
    \calE_k^e = \big\{ e \in \bbR^{n_x} : e^\top P_k^{-1}e \leq 1 \big\}, \quad \forall k \in \mathcal{N}^T_0.
\end{equation}
The volume of $\calE_k^\eta$ is proportional to $\det(Q_k)^{1/2}$; we use $\tr(Q_k)$ as a convex surrogate for this quantity.
\end{definition}
Before defining the feasibility of the resulting state funnel, we denote the feasible set for state at time $k$ as,
\begin{equation}
    \mathcal{F}_k = \{x: h(x) \leq0\}.
\end{equation}
\begin{definition}[Funnel feasibility and invariance]\label{def:invariance}
The state funnel sequence $(\calE_k^\eta)_{k=0}^T$, centered at $(\bar{x}_k)_{k=0}^T$, is \emph{feasible} if
\[
\bar{x}_k \oplus \calE_k^\eta \subseteq \mathcal{F}_k,
\qquad \forall k \in \mathcal{N}_0^T.
\]

The funnel pair $(\calE_k^\eta,\calE_k^e)$ is said to be \emph{deviation-dominant invariant} if, for every $k \in \mathcal{N}_0^{T-1}$,
$
\eta_k \in \calE_k^\eta, \;
e_k \in \calE_k^e,\;
V_k^e \geq \max\{\|v_k\|_2^2,\;\|w_k\|^2_2\},\;
V_k^\eta \geq \|w_k\|_2^2,
$
imply that
\[
\eta_{k+1} \in \calE_{k+1}^\eta,
\qquad
e_{k+1} \in \calE_{k+1}^e.
\]

Equivalently, once the state and observer deviations lie inside their respective funnels at time $k$, they remain inside the corresponding funnels at subsequent time steps, provided that the funnel levels dominate over the uncertainty.
\end{definition}

\begin{remark}
The lower-bound conditions in Definition~\ref{def:invariance} are introduced to exclude the ``uncertainty-dominant'' regime. If these conditions are removed, the funnel levels may become too small relative to the uncertainty, in which case the invariance may fail even when $(\eta_k,e_k)$ lies inside $(\calE_k^\eta,\calE_k^e)$. Thus, the definition only guarantees forward invariance when the funnel bounds are sufficiently large compared with the uncertainty level.
\end{remark}
\subsection{Incremental Stabilities}\label{subsec: quadratic stability}
Define the combined deviation states $z_k = [\eta_k^\top,e_k^\top]^\top$. Furthermore, for simplicity, we denote $V_k = z^\top_{k} \mbox{blkdiag}(Q^{-1}_{k}, P^{-1}_{k})z_{k}$, $V^e_k = e^\top_k P^{-1}_k e_k$, and $V_k^{\eta} = \eta^\top_k Q^{-1}_k \eta_k$. The \emph{quadratic incremental} stability conditions with $\sigma \in (0,1)$ and contraction rates $\alpha,\;\beta \in (0,1)$ are defined as the Lyapunov-like conditions,
\begin{equation}
\begin{aligned} \label{eq: quadratic stability}
    V^{\eta}_{k+1} &\leq \alpha V^{\eta}_k + \sigma V^e_k ,\quad V^\eta_k \geq \|w_k\|_2^2 \\
     \forall & (\delta p_k, \delta q_k) \in \{ (\delta p, \delta q): \|\delta p\|_2^2 \leq \gamma^2_k \|\delta q\|_2^2\} \\
    V_{k+1}^e &\leq \beta V_k^e , \quad V^e_k \geq \max \{\|v_k\|_2^2,\;\|w_k\|_2^2\} \\
    \forall & (\Delta p_k, \Delta q_k) \in \{(\Delta p, \Delta q): \| \Delta p\|^2_2 \leq\gamma^2_k \|\Delta q\|_2^2 \} \\
    0<&\beta + \sigma\leq \alpha <1,\quad \sigma +\alpha \leq 1.
\end{aligned}
\end{equation} 

The above condition can be interpreted as follows: whenever the system dynamics satisfies the set defined by the QCs arising from the nonlinearity and the corresponding lower-bound conditions on the Lyapunov functions, the incremental systems will be stable and invariant.  

\begin{prop} \label{prop 1}
Assume that for any $k\in \mathcal{N}_0^{T-1}$ such that $V^{\eta}_k\leq1$ and $V^e_k \leq1$, the system is in the \emph{deviation dominant} region with $V^{\eta}_k \geq \|w_k\|_2^2$, $V^e_k \geq \max \{\|v_k\|_2^2, \;\|w_k\|^2_2\}$.
Furthermore, assume that the Lipschitz conditions \eqref{eq:lips} and the following quadratic inequalities hold for some $\tau_x,\;\tau_y,\;v_x,\;v_y>0$ (via the S-procedure \cite{QC,bechet,taylor} on \eqref{eq: quadratic stability}):
\begin{equation}
\begin{aligned} \label{s-procedure a.}
    V^{\eta}_{k+1} &- \alpha V^{\eta}_k - \sigma V^e_k + \\\tau_x(&V^{\eta}_k - \|w_k\|_2^2)  + v_x(\gamma^2_k \|\delta q_k\|_2^2 - \|\delta p_k\|_2^2) \leq 0  ,
\end{aligned}
\end{equation}
and  
\begin{equation}
\begin{aligned}\label{s-procedure b.}
    V^{e}_{k+1} &- \beta V^{e}_k + \tau_x(V^{e}_k - \|w_k\|_2^2) + \\ \tau_y(&V^{e}_k - \|v_k\|_2^2)  + v_y(\gamma_k^2\|\Delta q_k\|_2^2 - \|\Delta p_k\|_2^2) \leq 0.
\end{aligned}
\end{equation}
Then the observer error contracts at $\beta$- rate and the combined system deviation contract at $\alpha$-rate with the funnel pair $(\calE_k^{\eta},\calE_k^e)$ being receding-invariant in the sense of Definitions \ref{def:funnel} and \ref{def:invariance}. 

\end{prop}

\begin{proof}
With positive scalars $\tau_x,\; v_x$ and conditions assumed, \eqref{s-procedure a.} implies that $V^{\eta}_{k+1} \leq \alpha V^{\eta}_k + \sigma V^e_k$. Similar arguments can be made to show that $\tau_x,\;\tau_y, \; v_y > 0$ together with \eqref{s-procedure b.} ensure that $V^e_{k+1} \leq \beta V^e_k \leq 1$ under the provided assumptions.

With the contraction Lyapunov function $V^e_k$ for the observer incremental system and $V^{\eta}_k, V^e_k \leq 1$, $V^{\eta}_{k+1} \leq \alpha V^{\eta}_k + \sigma V^e_k \leq \sigma + \alpha \leq 1$, the invariance of the control incremental system is ensured in the sense of \ref{def:invariance}. 

Lastly,  $\beta + \sigma \leq \alpha$ implies 
\[
\begin{aligned}
    V_{k+1} &= V^{\eta}_{k+1}  + V^e_{k+1}  \leq (\alpha V^{\eta}_k + \sigma V^e_k) + \beta V^e_k ,\\
    & = \alpha V^{\eta}_k + (\beta + \sigma) V^e_k  \leq \alpha(V^{\eta}_k + V^e_k) \leq \alpha V_k.
\end{aligned}
\] 
Thus the combined deviation contracts at $\alpha$ rate.
\end{proof}
\begin{remark}
Suppose that the observer contraction inequality in \eqref{eq: quadratic stability} together with the QC assumptions remain valid in the ``transient region'' prior to observer funnel entry.
Then, even if the initial estimation error lies outside the observer funnel defined in Definition~\ref{def:funnel}, the contraction of the observer incremental system implies that the observer deviation decreases and reaches the observer funnel after finitely many steps. Once the observer deviation enters the funnel, the same condition together with the lower-bound assumptions yields invariance for the observer funnel.

Moreover, if at the time of observer-funnel entry, the state deviation also satisfies the lower-bound condition and $V_k^{\eta}\leq 1$, then \eqref{eq: quadratic stability} implies that $V_{k+1}^{\eta}\leq 1$, that in turn, implies the invariance of the state funnel from that step onward in the deviation dominant region \ref{prop 1}. Thus, for large initial estimation errors, the funnel pair may be interpreted as eventually invariant, provided the dominance conditions continue to hold during the transient phase. For related discrete-time funnel invariance analysis, see~\cite{taewan,twewan1}.
\end{remark}
\section{Problem Statement}
\label{sec:problem}
A nonconvex conceptual optimization for solving the joint synthesis problem is now provided, that can be reformulated and solved as a sequence of convex subproblems in \S\ref{subsec: affine system}. Given the discrete-time dynamics \eqref{eq:dt-dynamics} and boundary conditions $(x_{\tiny \mbox{t0}}, x_{\tiny \mbox{tf}})$, design a reference trajectory $(\bar{x}_k,\bar{u}_k)_{k=0}^{T-1}$, gains $(K_k,L_k)_{k=0}^{T-1}$, and funnel shapes $(Q_k,P_k)_{k=0}^T$ such that: i) \eqref{eq:reference} holds; ii) the state funnel is feasible and the closed-loop deviation system \eqref{eq:deviation sytem} is funnel-invariant in the sense of Definition~\ref{def:invariance}; iii) the funnel-volume surrogate and nominal control effort are minimized. This leads to the following optimization problem.

\textbf{Problem 1: Nonconvex Joint Robust Guidance and Output-Feedback Synthesis Problem}
\begin{equation}
\begin{aligned} \label{problem: 1}
    \min_{\bar{\textbf{x}},\bar{\textbf{u}},\textbf{Q},\textbf{P},\textbf{K},\textbf{L}} &\sum_{k=0}^{T-1} \|\bar{u}_k\|_2^2 + \sum_{k=0}^T (\tr(Q_k) + \tr(P_k))  \\
    \text{s.t. }
    & \bar{x}_0 = x_{\tiny\mbox{t0}},\;\bar{x}_T = x_{\tiny \mbox{tf}}
    \\
    & \bar{x}_{k+1} = f_d(\bar{x}_k,\bar{u}_k,0),\; \forall k \in \mathcal{N}^{T-1}_0 \\
    & h(\bar{x}_k) < 0,\;\forall k \in \mathcal{N}^T_0\\
    & \bar{x}_k \oplus \calE_k^\eta\subseteq \mathcal{F}_k,\; \forall k \in \mathcal{N}^T_0 \\
    & Q_k \succ 0,\; P_k \succ 0,\;k \in \mathcal{N}^T_0\\
    & \text{Incremental stability and invariance conditions }\\ 
    & \text{in \eqref{eq: quadratic stability} hold for} \;\forall k \in \mathcal{N}^{T-1}_0,
\end{aligned}
\end{equation}
where $\bar{\textbf{x}} = (\bar{x})_{k=0}^T$, $\bar{\textbf{u}} = (\bar{u})_{k=0}^{T-1}$, $\textbf{Q} = (Q)_{k=0}^T$, $\textbf{P} = (P)_{k=0}^T$, $\textbf{K} = (K_k)_{k=0}^{T-1}$, and $\textbf{L} = (L)^{T-1}_{k=0}$ represent the concatenated variables.Let us designate the objective of this optimization problem as $F(\textbf{Z})$, and $\textbf{Z} = (\bar{\textbf{x}}, \bar{\textbf{u}}, \textbf{Q},\textbf{P},\textbf{K},\textbf{L})$ as the compact representation of all variables.

The joint design over both domains (guidance and robust feedback gains) allows the resulting reference trajectory to be robustly \emph{certifiable} in the sense of tracking stability and feasibility. 
Notice that this problem is nonconvex due to bilinearities among $(\bar{x}, \bar{u}, K, L, Q, P)$ and the nonlinear system dynamics. In the next section, we will reformulate the system constraints and stability conditions to enable an efficient solution by Interior Point Methods-based semidefinite programming (SDP) solvers.

\section{Joint Funnel/ Guidance Synthesis}
\label{sec:joint-funnel-synthesis}
This section provides a framework for formulating the feasibility and stability conditions stated in \S\ref{sec:prelim} as affine functions and LMIs, thereby yielding a local convex set representation of the constraints and conditions listed in \eqref{problem: 1}.

\subsection{Affine System Representation}
\label{subsec: affine system}
Recall that with the system Jacobians $A_k,B_k$, the undisturbed affine system assumes the form,
\begin{equation} \label{eq: affine system}
    \bar{x}_{k+1} +\delta x_{k+1} \approx f_d(\bar{x}_k,\bar{u}_k,0) + A_k \delta x_k + B_k \delta u_k.
\end{equation}
This affine representation of the system dynamics is important for approximating the nonconvex set described by nonlinear dynamics in terms
of an affine set in the incremental variables.

\subsection{LMI Surrogates for Funnel and System Feasibility}
\label{subsec:LMIs for feasibility}
From \cite{taylor,twewan1}, the feasibility conditions given in Definition \ref{def:invariance} can be formulated by the following quadratic matrix inequality (QMI),
\begin{equation}\label{eq:constraint_lmi_schur}
\begin{bmatrix}
    (b_k^x - a_k^{x,\top} \bar{x}_k)^2 & *\\
    Q_k a_{k}^x & Q_k
\end{bmatrix} \succeq 0 ,
\end{equation}
where $*$ denotes the transpose of the symmetric elements, $a^x_k = \nabla h(\bar{x}_k)$  is the gradient of the state constraints, $b^x_k = a^{x,\top}_k\bar{x}_k  - h(\bar{x}_k)$ are the constants for the half planes. 
Note that the bilinear terms, $(b_k - a_k^\top(\cdot))^2 = (\Phi_k(\cdot))^2$ can be approximated as $\Phi_k((\cdot)+\delta(\cdot))^2 \ \approx \Phi_k^2(\cdot) +2\Phi_k(\cdot) \delta\Phi_k$, and $\delta \Phi_k = -a_k^\top \delta(\cdot)$. Hence, the QMI can be cast into the standard LMI as, 
\begin{equation}\label{eq: constraints lmis}
\begin{aligned}
    &\begin{bmatrix}
    \Phi^{x,2}_k(\bar{x}_k) + 2\Phi_k^x(\bar{x}_k)\delta\Phi^x_k & *\\
    (Q_k+\delta Q_k) a_{k}^x & Q_k + \delta Q_k
\end{bmatrix}\succeq 0, \\
\end{aligned}
\end{equation}
with decision variables being the incremental quantities. 
\subsection{LMI Surrogates for Stability and Invariance}
\label{subsec:LMIs for stability}
We now derive the equivalent quadratic matrix inequality forms of \eqref{s-procedure a.} and \eqref{s-procedure b.} using Schur complements, thereby transforming the infinite-dimensional constraints into more ``tractable'' forms. In this direction, let us first define $A_k^{cl} = A_k + B_kK_k$, $B^{cl}_k = -B_kK_k$ for the closed incremental control system. 
Note that by expanding the terms, \eqref{s-procedure a.} can be written in the matrix form as,
\begin{equation}
\begin{aligned} \label{eq: control QMI 1}
    ((S_{\eta}^{\tiny\mbox{lin}} + S_{\eta}^{\tiny\mbox{QC}}) + M_\eta^\top Q^{-1}_{k+1}M_{\eta}) \preceq 0 ,
\end{aligned}
\end{equation}
where $M_{\eta} = [A^{cl}_k, B^{cl}_k, G_k,E_k]$, and 
\begin{equation}
\begin{aligned}
    S_{\eta}^{\tiny\mbox{lin}} & = 
    \begin{bmatrix}
        (\tau_x - \alpha)Q_k^{-1} & *&*&*\\
        0& -\sigma P^{-1}_k & *&* \\
        0&0& -\tau_xI &*\\
        0&0&0&0
    \end{bmatrix}, \\
    S_{\eta}^{\tiny\mbox{QC}} & = v_x \gamma ^2_k
    \begin{bmatrix}
         H^\top_{\delta,1} H_{\delta,1} & *& *& *\\
         H^\top_{\delta,2}H_{\delta,1} & H^\top_{\delta,2}H_{\delta,2} &*&* \\
         G^\top_qH_{\delta,1} &G^\top_qH_{\delta,2} &G^\top_qG_q    &*\\
        0&0&0&-\frac{I}{\gamma_k^2}
    \end{bmatrix},
\end{aligned}
\end{equation} with $H_{\delta,1} = C_q + D_qK_k$ and $H_{\delta,2} = -D_qK_k$.
Moreover, we note that \eqref{eq: control QMI 1} can be written equally as
\begin{equation}
\begin{bmatrix}
    (S_{\eta}^{\tiny\mbox{lin}} + S_{\eta}^{\tiny\mbox{QC}}) &* \\
    M_{\eta} & -Q_{k+1} 
\end{bmatrix}\preceq 0 .
\end{equation}
Since $S_{\eta}^{\tiny\mbox{lin}}$ contains $P^{-1}_k$ and $Q^{-1}_k$ terms, use the congruence with matrix $T_{\eta} = \mbox{blkdiag}(Q_k,P_k,I,I,I)$, we write \eqref{eq: control QMI 1} as 
\begin{equation}
    T^\top_{\eta}((S_{\eta}^{\tiny\mbox{lin}} + S_{\eta}^{\tiny\mbox{QC}}) + M_\eta^\top Q^{-1}_{k+1}M_{\eta}) T_{\eta} \preceq 0 ,
\end{equation}
with full QMI expansion as 
\begin{equation}
\begin{aligned}\label{eq: control QMI 2}
    &\begin{bmatrix}
        (\tau_x - \alpha)Q_k & *&*&*&*\\
        0&-\sigma P_k &*&*&* \\
        0&0& -\tau_x I &*&* \\
        0&0&0&0&* \\
        A^{cl}_kQ_k & B^{cl}_kP_k & G_k & E_k & -Q_{k+1}
    \end{bmatrix} +\\ 
    v_x \gamma ^2_k
    &\begin{bmatrix}
         H^{Q,\top}_{\delta} H_{\delta}^Q & *& *& *&*\\
        H_{\delta}^{P,\top}H^Q_{\delta} & H_{\delta}^{P,\top}H_{\delta}^{P} &*&* &*\\
         G^\top_qH^Q_{\delta} &G^\top_q H^P_{\delta} &G^\top_qG_q    &*&*\\
        0&0&0&-\frac{I}{\gamma_k^2}&* \\
        0&0&0&0&0&
    \end{bmatrix} \preceq 0 ,
\end{aligned}
\end{equation}
where $H^Q_{\delta} = H_{\delta,1}Q_k$ and $H^P_{\delta} = H_{\delta,2}P_k$.
Similarly, when applying the same procedural steps (Schur complement and congruence operation) to the \eqref{s-procedure b.}, the following quadratic matrix inequality for the incremental observer system is obtained:
\begin{equation}
\begin{aligned}\label{eq: observer QMI 1}
    &\begin{bmatrix}
        (\tau_x +\tau_y - \beta)P_k & *&*&*&*\\
        0&-\tau_x I&*&*&* \\
        0&0& -\tau_y I &*&* \\
        0&0&0&0&* \\
        A^{ob}_kP_k & G_k & L^{ob}_k & E_k & -P_{k+1}
    \end{bmatrix} +\\ 
    v_y \gamma ^2_k
    &\begin{bmatrix}
         H^{P,\top}_{\Delta} H_{\Delta}^P & *& *& *&*\\
        G_q^\top H_{\Delta}^P& G^\top_qG_q &*&* &*\\
        0 &0 &0   &*&*\\
        0&0&0&-\frac{I}{\gamma_k^2}&* \\
        0&0&0&0&0&
    \end{bmatrix} \preceq 0 ,
\end{aligned}
\end{equation}
where $A^{ob}_k = A_k - L_kC$, $L^{ob}_k = -L_kD$, and $H^P_{\Delta} = C_q P_k$. 
Note that \eqref{eq: control QMI 2} and \eqref{eq: observer QMI 1} still contain variable products in certain entries, including $A^{cl}_kQ_k$, $B^{cl}_kP_k$, $A^{ob}_k P_k$, $H^{Q,\top}_{\delta} H_{\delta}^Q$, $G^\top_qH^Q_{\delta}$, and $H^{P,\top}_{\Delta} H_{\Delta}^P$. To ensure the convexity of the feasible set defined by this inequality, the local affine approximation is performed for each of the product terms with small increments, leading to the formulation for each iterate,
\begin{equation}
\begin{aligned} \label{eq: affine variables}
    &(A^{cl}_k + \delta A^{cl}_k)(Q_k+\delta Q_k) \approx A^{cl}_kQ_k  + B_k\delta K_kQ_k + A^{cl}_k \delta Q_k, \\
    &(A^{ob}_k + \delta A^{ob}_k)(P_k + \delta P_k) \approx A^{ob}_kP_k  - \delta L_kCP_k + A^{ob}_k \delta P_k, \\
    &(B^{cl}_k + \delta B^{cl}_k)( P_k + \delta P_k) \approx B^{cl}_kP_k - B_k \delta K_kP_k - B_kK_k \delta P_k, \\
    &(H^{P}_{\delta} + \delta H^P_{\delta})^\top(H^{Q}_{\delta} + \delta H^Q_{\delta}) \approx H^{P,\top}_{\delta}H^{Q}_{\delta}  +\\& \delta H^{P,\top}_{\delta}H^{Q}_{\delta} + H^{P,\top}_{\delta} \delta H^{Q}_{\delta} ,\\
    & (H^Q_{\delta} + \delta H^Q_{\delta})^\top(H^Q_{\delta} + \delta H^Q_{\delta}) \approx H^{Q,\top}_{\delta}H^Q_{\delta} + \\&H^{Q,\top}_{\delta}\delta H^Q_{\delta} + \delta H^{Q,\top}_{\delta} H^Q_{\delta} ,\\
    & (H_{\delta}^P + \delta H_{\delta}^P)^\top(H_{\delta}^P + \delta H_{\delta}^P) \approx H^{P,\top}_{\delta}H^P_{\delta} +\\ & \delta H^{P,\top}_{\delta}H^P_{\delta} + H^{P,\top}_{\delta}\delta H^P_{\delta},\\
    & (H^P_{\Delta} + \delta H^P_{\Delta})^\top(H^P_{\Delta} + \delta H^P_{\Delta}) \approx H^{P,\top}_{\Delta}H^P_{\Delta} + \\ & \delta H^{P,\top}_{\Delta}H^P_{\Delta} + H^{P,\top}_{\Delta}\delta H^P_{\Delta},
\end{aligned}
\end{equation} with $\delta H^Q_{\delta} \approx D_q \delta K_kQ_k +(C_q +D_qK_k)\delta Q_k$, $\delta H^P_{\delta} \approx -D_q(\delta K_k P_k + K_k \delta P_k)$, and $\delta H^P_{\Delta} \approx C_q \delta P_k$.
Similar to \eqref{eq: constraints lmis}, we now consider the incremental form of these matrix inequalities, which are LMIs in terms of the incremental quantities.
Therefore, the stability and invariance stated in Proposition \ref{prop 1} are now expressed in terms of LMIs from the local affine approximations of \eqref{eq: control QMI 2} and \eqref{eq: observer QMI 1} through \eqref{eq: affine variables}. Furthermore, if the approximation residue is small (that usually occurs when the solution converges to the feasible ones), then \eqref{eq: control QMI 2} and \eqref{eq: observer QMI 1} will be close to exact through the affine approximation.
\begin{remark}
Notice that the matrix inequalities for stability analysis \eqref{eq: control QMI 2} and \eqref{eq: observer QMI 1} are separated into two parts: Linear system and nonlinear QC parts. If the system described in \eqref{eq:ct-dynamics} is linear, then one can simply set $E = 0$ and $v_x = v_y = 0$ to recover the expression for the linear system. Moreover, unlike the classical linear case where the separation principle applies, the coupling nature of two funnels necessitates the joint design of both funnels and the corresponding gains.
\end{remark}

\subsection{Lipschitz Constant Estimation}\label{subsec: lipschitz}
As the above QCs are enforced through Lipschitz conditions, we must estimate local Lipschitz constants. Discretization can make these constants difficult to express analytically, even when the continuous-time bounds are readily available. Therefore, we adopt the M-problem in $\gamma$- iteration of \cite{taylor_phd} to compute accurate set-wise estimates via sampling. Unlike \cite{taylor_phd} however, we sample a larger set that reflects both sensing and control schemes. Define
\begin{equation} \label{eq: est set}
    \calE_k = \{\bar{x}_k + \delta x:\;\delta x^\top O^{-1}_k \delta x \leq 1\},
\end{equation}
where $O_k =(\sqrt{\lambda_{\tiny \mbox{max}}(Q_k)} + \sqrt{\lambda_{\tiny\mbox{max}}(P_k)})^2 I$; hence, $\calE_k$ is an n-dimensional sphere centered at the reference trajectory with radius determined by the \emph{largest} eigenvalues of the control- and observation-funnel shape matrices with samples taken uniformly on the boundary of $\calE_k$. Finally, we define the sampled Lipschitz constants over the trajectory by $\Gamma = (\gamma_k)_{k=0}^T$. 
As the Lipschitz constant is estimated through the sampling, \S\ref{prop 1} holds empirically, unless a global Lipschitz condition is used, which is, in general, overly conservative.

\subsection{Sequential Convex Programming}
\label{subsec: SCP}
In this section, we use a prox-type SCP methodology~\cite{Moreau_regularization, scvx_main} to solve \eqref{problem: 1}; we briefly discuss the approach below.

Recall that the nonconvex constraints in \eqref{problem: 1} are locally approximated by convex surrogate sets: the affine surrogate \eqref{eq: affine system} for the system dynamics, the LMI surrogate \eqref{eq: constraints lmis} for funnel and system feasibility, and the LMI surrogates of \eqref{eq: control QMI 2} and \eqref{eq: observer QMI 1} for the state and observer funnel stability conditions. These surrogates are constructed with incremental variables $\delta \bar{x}_k,\; \delta \bar{u}_k$ for the system dynamics, $\delta Q_k,\; \delta P_k$ for the funnel shapes, and $\delta K_k,\; \delta L_k$ for the state/output-feedback gains. The prox-SCP method with proximal parameter $\lambda$ is then used to formulate the following convex subproblem for efficiently solving the robust guidance and output-feedback funnel synthesis problem with the operator $\Theta(a_1,\ldots,a_I) = \sum_{i=1}^I \|a_i\|^2$, where $\|\cdot\|$ denotes the $\ell_2$ norm when $a_i$ is a vector and the Frobenius norm when $a_i$ is a matrix.

\textbf{Problem 2: Convex Joint Robust Guidance and Output-Feedback Synthesis Subproblem}
\begin{equation}
\begin{aligned} \label{problem: 2}
    \min_{\delta \textbf{Z}} &\sum_{k=0}^{T-1} \|\bar{u}_k+\delta \bar{u}_k\|_2^2 + \sum_{k=0}^T (\tr(Q_k + \delta Q_k) + \tr(P_k + \delta P_k)) +\\
    & \frac{1}{2 \lambda}\Theta(\delta \textbf{Z})  \\
    \text{s.t. }
    & \bar{x}_0 + \delta\bar{x}_0 = x_{\tiny \mbox{t0}},\; \bar{x}_T + \delta\bar{x}_T = x_{\tiny \mbox{tf}}\\
    & Q_k + \delta Q_k\succ0,\; P_k + \delta P_k \succ 0,\; \forall k \in \mathcal{N}^T_0\\
    & \text{Affine surrogate of \eqref{eq: affine system} holds for system dynamics,}  \\
    &\forall k \in \mathcal{N}^{T-1}_0\\
    & \text{LMI surrogate of \eqref{eq: constraints lmis} holds for state funnel feasibility,} \\
    &\forall k \in \mathcal{N}^{T}_0\\
    & \text{LMI surrogate of \eqref{eq: control QMI 2} \eqref{eq: observer QMI 1} hold for state and } \\
    & \text{observer funnel stabilities, } \forall k \in \mathcal{N}^{T-1}_0,
\end{aligned}
\end{equation}
where $\delta\textbf{Z} = (\delta\bar{\textbf{x}},\delta\bar{\textbf{u}},\delta \textbf{Q},\delta \textbf{P}, \delta \textbf{K}, \delta \textbf{L})$, and each element in $\delta \textbf{Z}$ represents the concatenation of the corresponding incremental variable. Similar to \eqref{problem: 1}, we also define the objective of this problem as $F_{\tiny \mbox{prox}} (\delta \textbf{Z})$. Lastly, we denote the function $F_{\tiny \mbox{cvx}}(\delta \textbf{Z}) = F_{\tiny\mbox{prox}} - \frac{1}{2 \lambda}\Theta(\delta \textbf{Z})$ as the objective function without the regularization term. Note that since the convex subproblem uses LMIs and affine surrogates, the interim solutions may not satisfy the original nonconvex constraints. We also acknowledge that, due to the nature of SCP methods, feasibility is often difficult to guarantee, and a post-solution check is required. For the feasibility of the general convex surrogate methods, please refer to \cite{ch10}.

\section{Algorithm}
\label{sec:algorithm}
A concise iterative algorithm for solving \eqref{problem: 2} will be given with Armijo-style backtracking updating methods with a varying proximal-parameter. 

Define the reduction ratio 
\begin{equation}\label{eq: ratio}
    r = \frac{F(\textbf{Z}^I)- F(\textbf{Z}^I + \delta \textbf{Z})}{|F(\textbf{Z}^I) - F_{\tiny \mbox{cvx}}(\delta\textbf{Z})|},
\end{equation} and the change in the objective values $\Delta F = F(\textbf{Z}^{I+1}) - F(\textbf{Z}^I)$.
The algorithm for solving \eqref{problem: 1} by solving \eqref{problem: 2} iteratively is given as 
\begin{algorithm}[t]
\caption{Prox-SCP algorithm}
\label{alg: prox scp}
\begin{algorithmic}[1]
    
    \State \textbf{Input:} Set the proper initial variables $\textbf{Z}^0$, choose QC constants $\tau_x,\tau_y,v_x,v_y$,  set $\Delta F>\epsilon$, and the desired convergence rate $\alpha, \beta$ that satisfy \eqref{eq: quadratic stability}.
    \While{$|\Delta F | > \epsilon$}
        \State Solve M-problems \cite{taylor_phd} over sets \eqref{eq: est set} to get $\Gamma$
        \State Solve Problem (\ref{problem: 2}) to get $\delta\textbf{Z}$.
        \State Compute $r$ from \eqref{eq: ratio}.
        \If{$r \geq r_{\tiny \mbox{min}}$}
            \State $\textbf{Z}^{I+1} \gets \textbf{Z}^I + \delta \textbf{Z}$
        \Else
            \State Step Rejection 
            \State $\lambda = \omega \lambda$
        \EndIf
        \State $I \gets I +1$ and update $\Delta F$ 
    \EndWhile
    \State \textbf{return} $\textbf{Z}^I$
\end{algorithmic}
\end{algorithm}
with the desired cost reduction ratio $r_{\tiny \mbox{min}} \in (0,1)$, parameter shrinking ratio $\omega \in (0,1)$. If the convex subproblem is  initially infeasible, slack variables can be introduced or increased until a feasible solution is obtained. For more details on penalty methods, please refer to \cite{ch10,hypersonic}. Lastly, the QC constants $\tau_x,\; \tau_y,\; v_x ,\; v_y$ can be chosen to be small $\approx 0.1$ if the user is confident about the uncertainty models for a less restrictive behavior. However, the selection of these hyperparameters still suffers from the drawbacks of most funnel-based approaches: they must be hand-tuned by users \cite{taylor,taylor_phd,twewan1}.

\section{Numerical Simulations}
\label{sec:numerical-simulations}
Consider the unicycle model evolving under system uncertainties, including random disturbances in the system states and noisy position estimates (simulating GPS errors). 
\begin{equation}
\begin{aligned}
    \begin{bmatrix}
        \dot{x}_1 \\ \dot{x}_2 \\ 
        \dot{x}_3
    \end{bmatrix}
     & = 
    \begin{bmatrix}
        u_1 \cos{x_3} \\
        u_1 \sin{x_3}\\
        u_2
    \end{bmatrix} + 
    \begin{bmatrix}
        0.1 & 0 \\
        0 & 0.5 \\
        0 & 0 
    \end{bmatrix}
    \begin{bmatrix}
        w_1 \\ w_2
    \end{bmatrix}, 
    \\
    \begin{bmatrix}
        y_1 \\ y_2 
    \end{bmatrix}
    &= 
    \begin{bmatrix}
        1 & 0 & 0\\
        0 & 1 & 0
    \end{bmatrix}
    \begin{bmatrix}
        x_1 \\ x_2 \\ x_3
    \end{bmatrix} +
    \begin{bmatrix}
        0.3 & 0 \\
        0 & 0.1 
    \end{bmatrix}
    \begin{bmatrix}
        v_1 \\ v_2
    \end{bmatrix},
\end{aligned}
\end{equation}
where $x_1,x_2$ denote the x,y-position of the robots, $x_3$ represents the heading angle, $u_1$ is the linear velocity, and $u_2$ serves as the angular velocity. The state constraint for obstacle avoidance is 
$h(x) = r^2 -\|[x_1,x_2]^\top- [x_{\tiny \mbox{obs}}^i,y_{\tiny \mbox{obs}}^i]^\top\|^2_2  \leq 0 $
with $i$-th obstacle center $[x_{\tiny \mbox{obs}}^i,y_{\tiny \mbox{obs}}^i]^\top$.
Note that all uncertainties are normalized and scaled through the corresponding injection channels, and the asymmetrical uncertainty input channels are selected to highlight the direction-dependent effect of uncertainties. Table \ref{table: parameters} summarizes the parameters used in the simulation, with a total of four cases simulated over joint and decoupled methods.

\begin{table}[t]
\caption{Simulation Parameters}
\label{table: parameters}
\begin{center}
\begin{tabular}{ccccc}

Parameters & Case 1 & Case 2 & Case 3 & Case 4
\\\hhline{=====}
Uncertainty & System  & Observer  & System and & System and \\
Type&&&Observer& Observer
\\ \hline
IC (States) & Uniform  on & $[0,0,0]$ & Uniform on & $0.5[\cos{\frac{9 \pi}{10}},$\\

& the funnel&& the funnel&$\sin{\frac{9 \pi}{10}},0]$\\

\hline
IC (Estimate) & $[0,0,0]$ &Uniform in& $[0,0,0]$ & Uniform in\\
&  & the funnel &  & the funnel 

\\\hhline{=====}
\multicolumn{5}{p{242pt}}{
Obstacles are centered at $[4,3]^\top$ and $[9,3]^\top$ m with
$r=1$ m. Forward-Euler discretization uses $dt=0.067$ s.
$\alpha=0.98$, $\beta=0.8$, and
$\tau_x=\tau_y=v_x=v_y=0.1$.
The maximum eigenvalues QMI violations for the controller and observer are
\eqref{eq: control QMI 2}--\eqref{eq: observer QMI 1}
are $1.35 \times10^{-3}/2.72\times 10^{-4}$ for the joint methods.
}
\end{tabular}
\end{center}
\end{table}

In the first two cases, we isolated the process and measurement uncertainties to study the individual stability properties of each funnel. In the first case, we uniformly sample initial conditions (ICs) around the entry boundary of the process-uncertainty state funnel, while in the second case, we sample them around the measurement-uncertainty estimation funnel. Both cases show rapid convergence of the sampled trajectories, which remain well within the funnel boundaries as shown in Fig. \ref{fig:funnel comparison}.

The third and fourth cases demonstrate the controller's invariance and the observer's convergence, respectively. To distinguish the state and observation funnels, in Fig. \ref{fig:combined funnels}, the state funnels are in gray and
observer funnels are in dark blue. 
For each \emph{reference} trajectory, we define a single state funnel based on the state difference $\eta = x  -\bar{x}$. For each \emph{actual} trajectory, we also define one observation funnel, since observation funnels bound the deviation $e = x - \hat{x}$. Together, these two cases reflect the stability and conditional invariance in \S\ref{sec:problem}.

Despite the visual similarity between decoupled and joint methods, Fig. \ref{fig:states bounds} shows a non-negligible invariance-bound violation, while the proposed joint framework, under the same uncertainty and ICs, satisfies the invariance conditions. In general, bilevel methods are prone to convergence issues even in the biconvex setting \cite{bilevel1}. The code for producing the results can be found at: https://github.com/Justin900308/Demo-code-for-joint-funnel-observer-and-controller-synthesis

\begin{figure}
\centering
\includegraphics[width=0.46\textwidth]{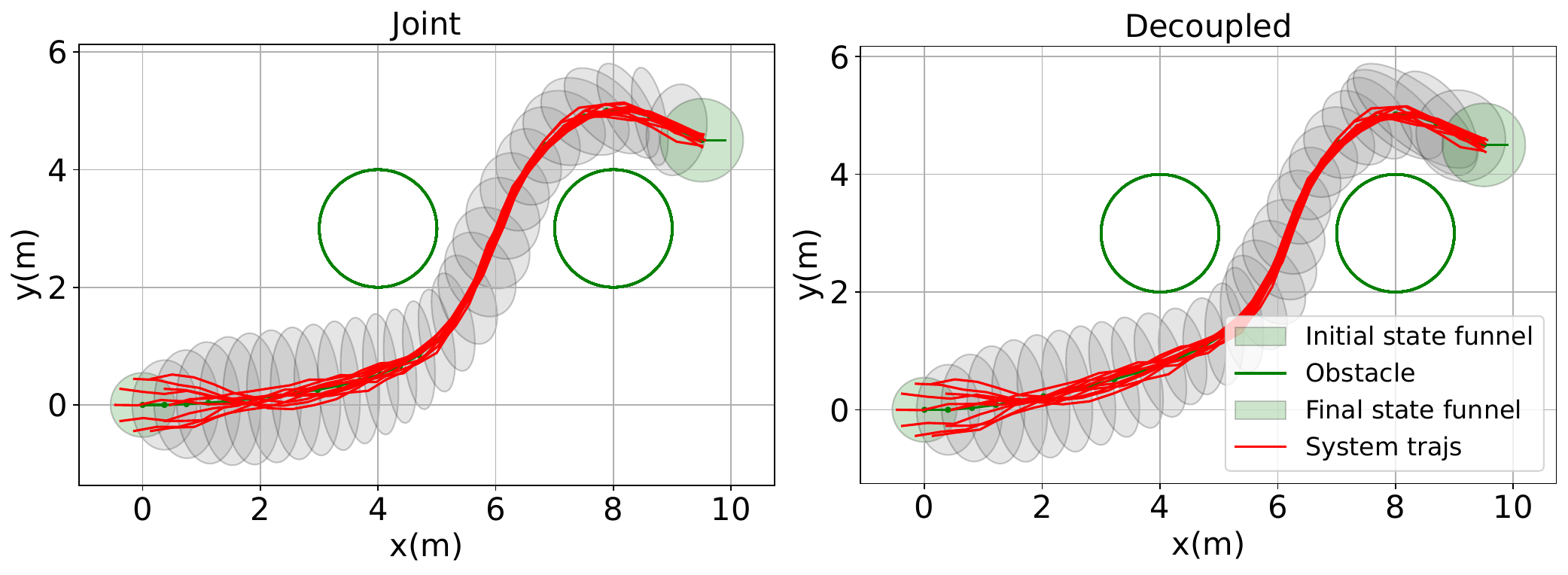}
\includegraphics[width=0.46\textwidth]{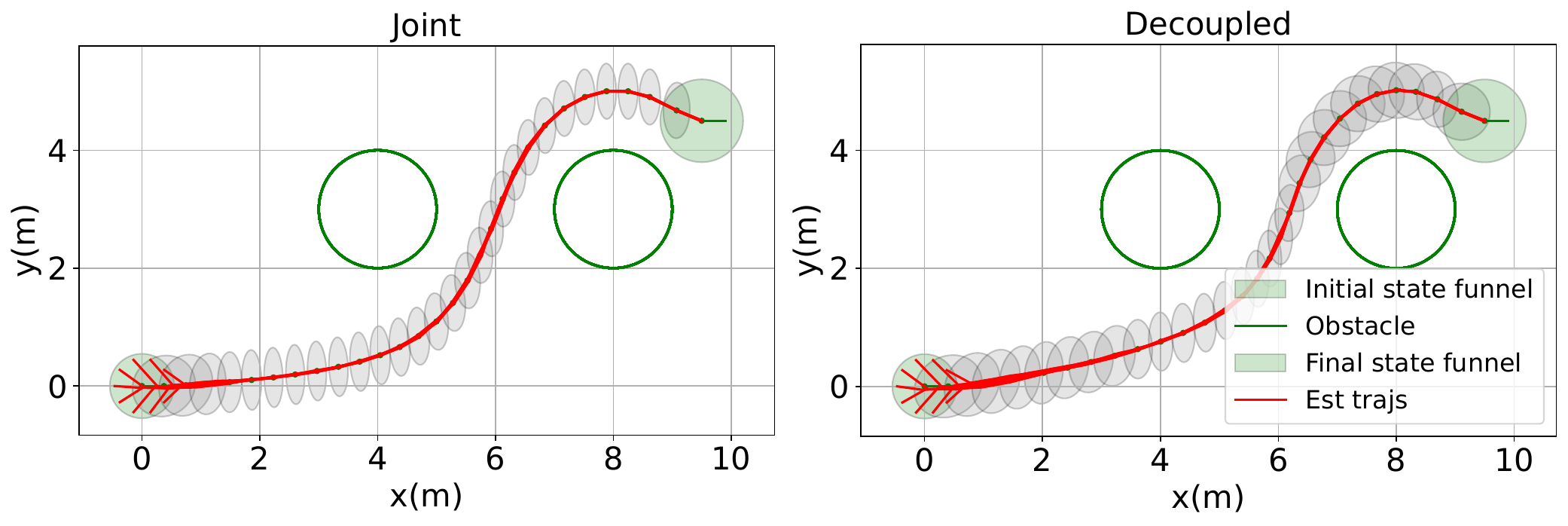}
\caption{State and observer funnels with isolated effects for case 1 (upper) and case 2 (lower) over two approaches.}
\label{fig:funnel comparison}
\end{figure}

\begin{figure}
\centering
\includegraphics[width=0.46\textwidth]{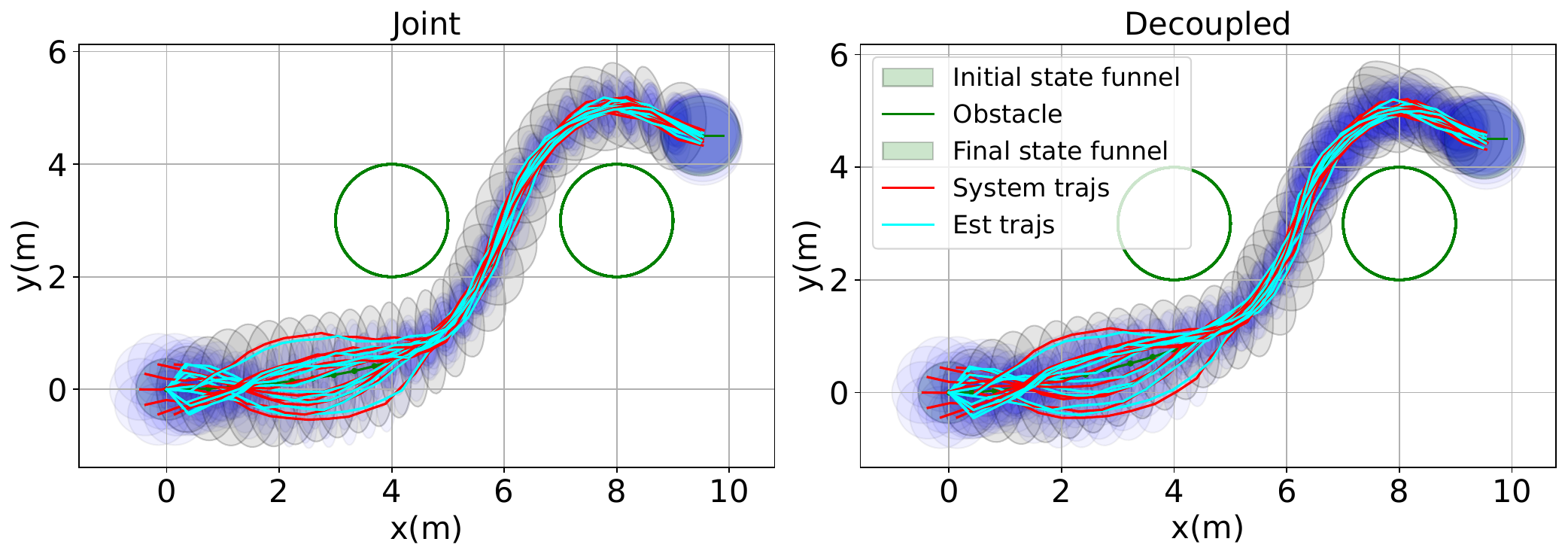}
\includegraphics[width=0.46\textwidth]{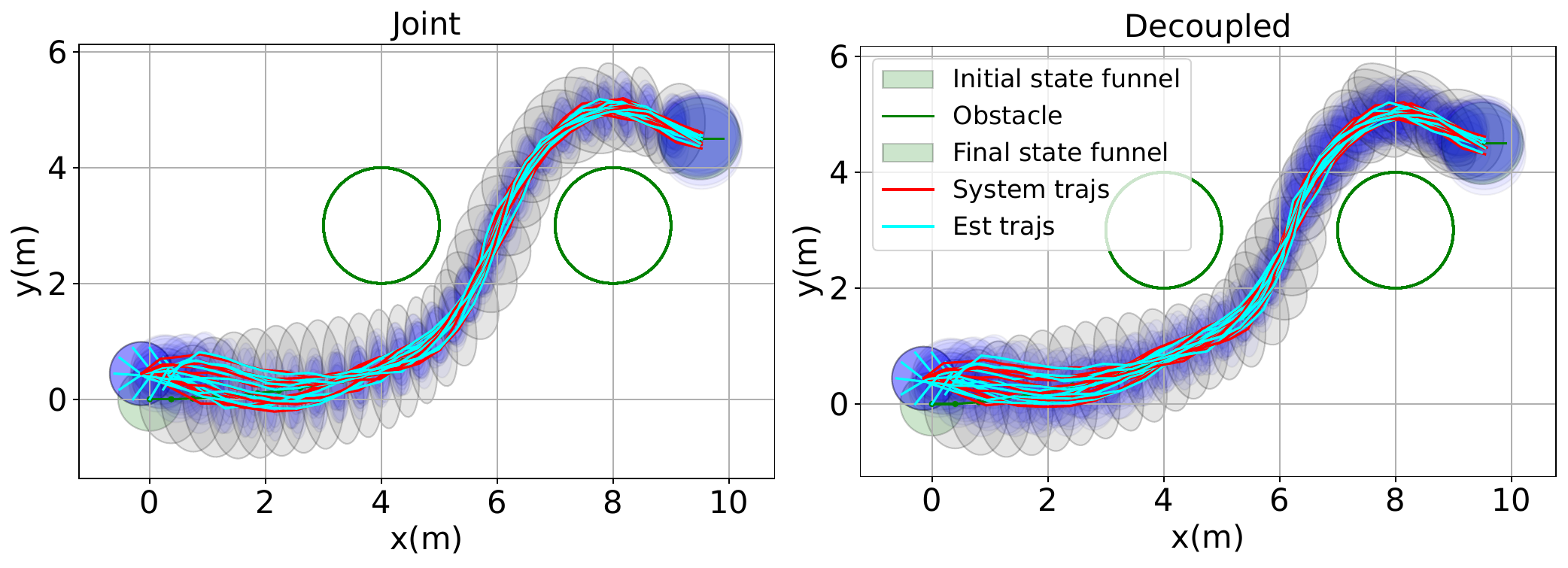}
\caption{Combined state and observer funnels for case 3 (upper), and case 4 (lower) over two approaches.}
\label{fig:combined funnels}
\end{figure}

\begin{figure}
\centering
\includegraphics[width=0.45\textwidth]{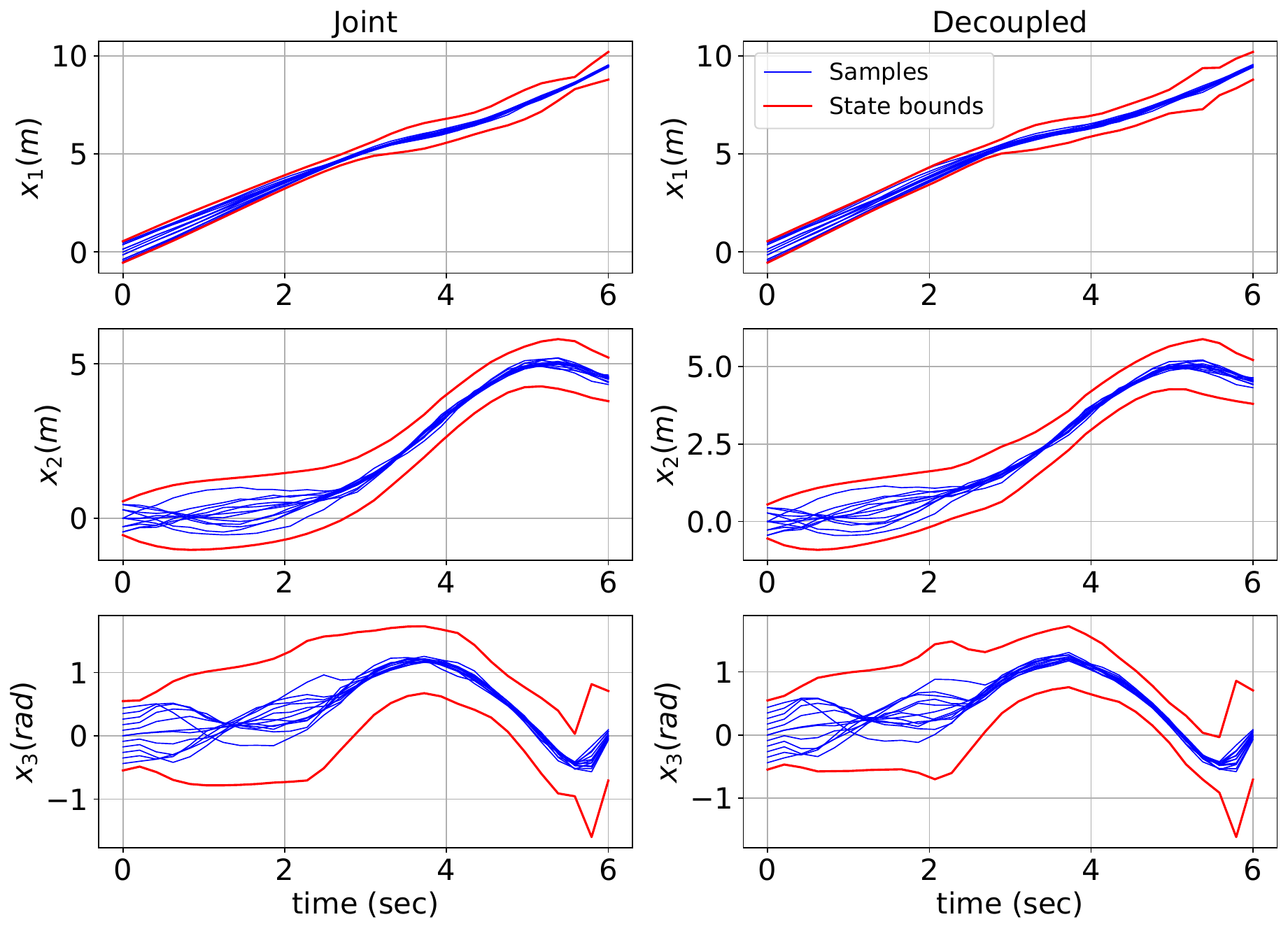}
\includegraphics[width=0.45\textwidth]{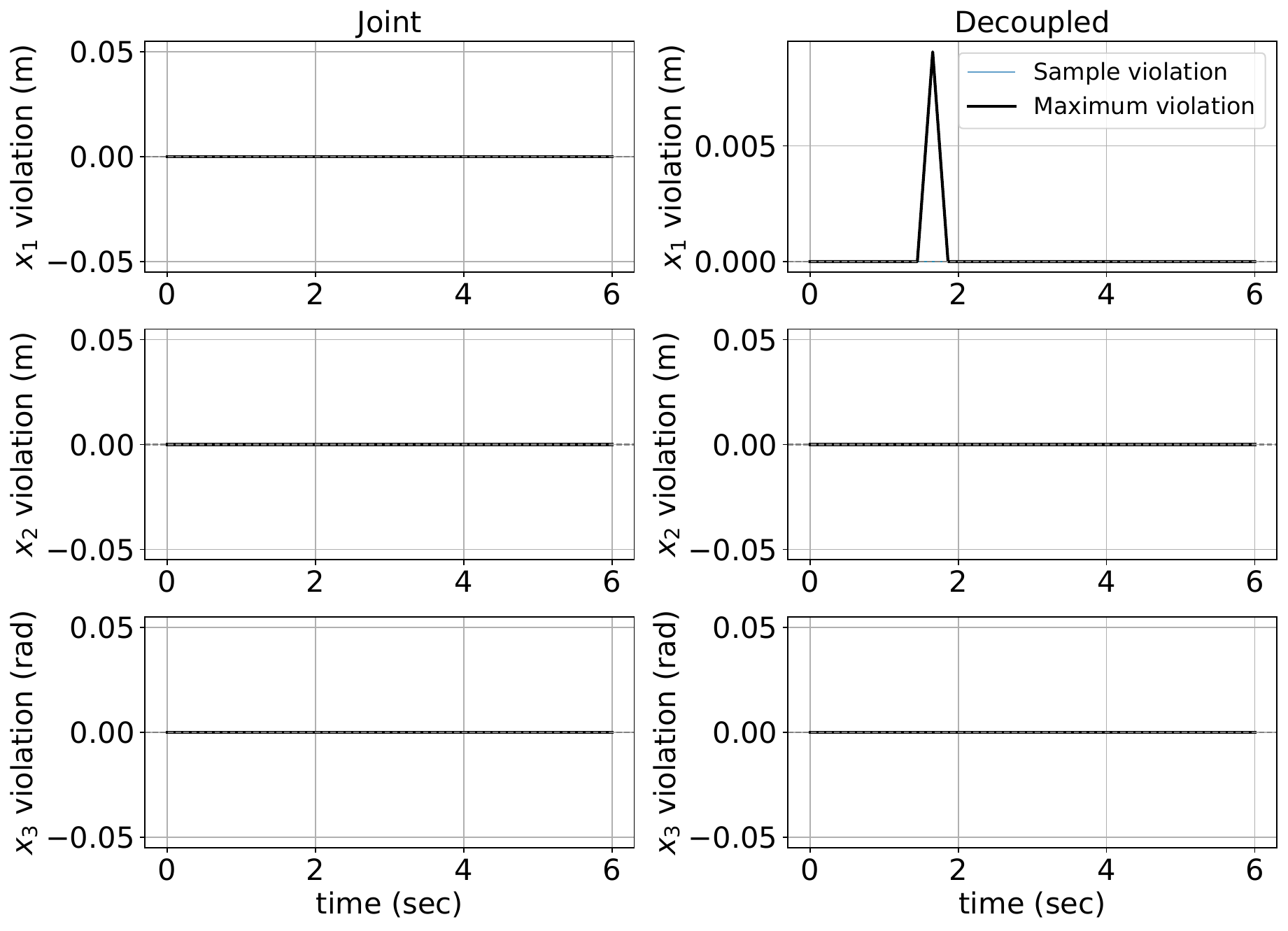}
\caption{State bounds and invariance-violation time histories for case 3 under two approaches.}
\label{fig:states bounds}
\end{figure}

\section{Concluding}
\label{sec:conclusion}
We present a unified framework for jointly synthesizing the guidance and robust observer-based output-feedback problems by combining LMI-based Lyapunov stability and SCP techniques. Compared with prior works that solve this problem in two or three steps (guidance-robust feedback, robust observer), this work mitigates the feasibility issues of traditional approaches by considering the joint synthesis as a unified design problem. In this direction, by explicitly expressing the state and observer funnels, one can establish the invariance properties separately, in contrast to traditional approaches that use a single funnel accounting for both state and estimation deviations. The resulting funnels can also serve as a tool to examine the worst-case scenarios for a reference trajectory and the corresponding feedback gains.
Lastly, numerical results on a unicycle system demonstrate the effectiveness of the proposed framework to address the presence of process and sensing uncertainties in the guidance, observer, and control design problems.
\bibliographystyle{ieeetr}
\bibliography{references}

@book{rudin,
    author = {Walter Rudin},
    title = {Principles of Mathematical Analysis} ,
    publisher = {McGraw Hill},
    year = {1953}
}

@phdthesis{taewan,
    author = {Taewan Kim} ,
    title = {Robust Predictive Control for Uncertain Nonlinear Systems via Funnel Synthesis},
    school = {University of Washington},
    year = {2025}
}

@article{bilevel1,
author = {Gorski, Jochen and Pfeuffer, Frank and Klamroth, Kathrin},
year = {2007},
month = {11},
pages = {373-407},
title = {Biconvex Sets and Optimization with Biconvex Functions: A survey and extensions},
volume = {66},
journal = {Mathematical Methods of Operations Research},
doi = {10.1007/s00186-007-0161-1}
}

@article{taylor,
author = {Taylor Reynolds and Danylo Malyuta and Mehran Mesbahi and Behcet Acikmese and John M. Carson},
title = {Funnel Synthesis for the 6-{DOF} Powered Descent Guidance Problem},
journal = {AIAA Scitech 2021 Forum},
year = {2021},
doi = {10.2514/6.2021-0504}
}

@phdthesis{taylor_phd,
    author = {Taylor Patrick Reynolds},
    title = {Computational Guidance and Control for Aerospace Systems},
    school = {University of Washington},
    year = {2021}
}

@article{hypersonic,
author = {Mceowen, Skye and Calderone, Daniel J. and Tiwary, Aman and Zhou, Jason S. K. and Kim, Taewan and Elango, Purnanand and A\c{c}\i{}kme\c{s}e, Beh\c{c}et},
title = {Autotuned Primal–Dual Successive Convexification for Reentry Guidance},
journal = {Journal of Guidance, Control, and Dynamics},
volume = {48},
number = {9},
pages = {2020-2040},
year = {2025},
doi = {10.2514/1.G008692}
}

@incollection{ch10,
author = {Bazaraa, M.S. and Sherali, H.D. and Shetty, C.M.},
publisher = {John Wiley and Sons, Ltd},
isbn = {9780471787778},
title = {Nonlinear Programming},
chapter = {10: Methods of Feasible Directions},
pages = {537-653},
doi = {10.1002/0471787779.ch10},
year = {2006}
}

@ARTICLE{scvx_main,
  author={Malyuta, Danylo and Reynolds, Taylor P. and Szmuk, Michael and Lew, Thomas and Bonalli, Riccardo and Pavone, Marco and Açıkmeşe, Behçet},
  journal={IEEE Control Systems Magazine}, 
  title={Convex Optimization for Trajectory Generation: A Tutorial on Generating Dynamically Feasible Trajectories Reliably and Efficiently}, 
  year={2022},
  volume={42},
  number={5},
  pages={40-113},
  doi={10.1109/MCS.2022.3187542}}

@article{Moreau_regularization,
author = {Lemar\'{e}chal, Claude and Sagastiz\'{a}bal, Claudia},
title = {Practical Aspects of the Moreau--Yosida Regularization: Theoretical Preliminaries},
journal = {SIAM Journal on Optimization},
volume = {7},
number = {2},
pages = {367-385},
year = {1997},
doi = {10.1137/S1052623494267127}
}

@article{twewan1,
  title={Joint synthesis of trajectory and controlled invariant funnel for discrete‐time systems with locally {L}ipschitz nonlinearities},
  author={Taewan Kim and Purnanand Elango and Behçet Açikmese},
  journal={International Journal of Robust and Nonlinear Control},
  year={2024},
  volume={34},
  pages={4157 - 4176},
  url={https://api.semanticscholar.org/CorpusID:252118592}
}

@article{bechet,
title = {Stability analysis with quadratic {L}yapunov functions: Some necessary and sufficient multiplier conditions},
journal = {Systems and Control Letters},
volume = {57},
number = {1},
pages = {78-94},
year = {2008},
issn = {0167-6911},
doi = {https://doi.org/10.1016/j.sysconle.2007.06.018},
url = {https://www.sciencedirect.com/science/article/pii/S0167691107000941},
author = {Behçet Açıkmeşe and Martin Corless},
}

@inbook{QC,
author = {Stephen Boyd and Laurent El Ghaoui and Eric Feron and Venkataramanan Balakrishnan},
title = {Linear Matrix Inequalities in System and Control Theory},
chapter = {8. Lur'e and Multiplier Methods},
pages = {119-129},
doi = {10.1137/1.9781611970777.ch8},
URL = {https://epubs.siam.org/doi/abs/10.1137/1.9781611970777.ch8},
eprint = {https://epubs.siam.org/doi/pdf/10.1137/1.9781611970777.ch8},
year = {1994}
}

@article{majumdar_tedrake,
  author    = {Anirudha Majumdar and Russ Tedrake},
  title     = {Funnel Libraries for Real-Time Robust Feedback Motion Planning},
  journal   = {The International Journal of Robotics Research},
  volume    = {36},
  number    = {8},
  pages     = {947--982},
  year      = {2017},
  doi       = {10.1177/0278364917712421}
}

@article{leeman,
  author    = {Antoine P. Leeman and Johannes K\"{o}hler and Andrea Zanelli
               and Samir Bennani and Melanie N. Zeilinger},
  title     = {Robust Nonlinear Optimal Control via System Level Synthesis},
  journal   = {IEEE Transactions on Automatic Control},
  year      = {2025},
}

@article{manchester_kuindersma,
  author    = {Ian R. Manchester and Scott Kuindersma},
  title     = {Robust Direct Trajectory Optimization Using Approximate
               Invariant Funnels},
  journal   = {Autonomous Robots},
  volume    = {43},
  number    = {2},
  pages     = {375--387},
  year      = {2019},
  doi       = {10.1007/s10514-018-9779-5}
}

@article{atassi_khalil,
  author    = {A. N. Atassi and H. K. Khalil},
  title     = {A Separation Principle for the Stabilization of a Class of
               Nonlinear Systems},
  journal   = {IEEE Transactions on Automatic Control},
  volume    = {44},
  number    = {9},
  pages     = {1672--1687},
  year      = {1999},
  doi       = {10.1109/9.788534}
}

@article{mayne_tube,
  author    = {David Q. Mayne and Mar{\'i}a M. Seron and Sa{\v{s}}a V.
               Rakovi{\'c}},
  title     = {Robust Model Predictive Control of Constrained Linear Systems
               with Bounded Disturbances},
  journal   = {Automatica},
  volume    = {41},
  number    = {2},
  pages     = {219--224},
  year      = {2005},
  doi       = {10.1016/j.automatica.2004.08.019}
}

@article{mayne_of,
  author    = {David Q. Mayne and Sa{\v{s}}a V. Rakovi{\'c} and Rolf
               Findeisen and Frank Allg\"{o}wer},
  title     = {Robust Output Feedback Model Predictive Control of Constrained
               Linear Systems},
  journal   = {Automatica},
  volume    = {42},
  number    = {7},
  pages     = {1217--1222},
  year      = {2006},
  doi       = {10.1016/j.automatica.2006.03.005}
}

@article{xu_obs,
  author    = {Xiangru Xu and Beh{\c{c}}et A{\c{c}}{\i}kmese and
               Martin J. Corless},
  title     = {Observer-Based Controllers for Incrementally Quadratic
               Nonlinear Systems with Disturbances},
  journal   = {IEEE Transactions on Automatic Control},
  volume    = {66},
  number    = {3},
  pages     = {1129--1143},
  year      = {2021},
  doi       = {10.1109/TAC.2020.2996985}
}

@article{ccm,
  author    = {Ian R. Manchester and Jean-Jacques E. Slotine},
  title     = {Control Contraction Metrics: Convex and Intrinsic Criteria
               for Nonlinear Feedback Design},
  journal   = {IEEE Transactions on Automatic Control},
  volume    = {62},
  number    = {6},
  pages     = {3046--3053},
  year      = {2017},
  doi       = {10.1109/TAC.2017.2668380}
}

@inproceedings{singh_contraction,
  author    = {Sumeet Singh and Anirudha Majumdar and Jean-Jacques E. Slotine
               and Marco Pavone},
  title     = {Robust Online Motion Planning via Contraction Theory and
               Convex Optimization},
  booktitle = {Proceedings of the IEEE International Conference on Robotics
               and Automation (ICRA)},
  pages     = {5883--5890},
  year      = {2017},
  doi       = {10.1109/ICRA.2017.7989693}
}

@article{luenberger,
  author  = {David G. Luenberger},
  title   = {Observers for Multivariable Systems},
  journal = {{IEEE} Transactions on Automatic Control},
  volume  = {11},
  number  = {2},
  pages   = {190--197},
  year    = {1966},
  doi     = {10.1109/TAC.1966.1098323}
}

@article{acikmese_ploen,
  author  = {Beh{\c{c}}et A{\c{c}}{\i}kme{\c{s}}e and Scott R. Ploen},
  title   = {Convex Programming Approach to Powered Descent Guidance
             for {Mars} Landing},
  journal = {Journal of Guidance, Control, and Dynamics},
  volume  = {30},
  number  = {5},
  pages   = {1353--1366},
  year    = {2007},
  doi     = {10.2514/1.27553}
}

@article{mayne_mpc,
  author  = {David Q. Mayne and James B. Rawlings and Christopher V. Rao
             and Pierre O. M. Scokaert},
  title   = {Constrained Model Predictive Control: Stability and Optimality},
  journal = {Automatica},
  volume  = {36},
  number  = {6},
  pages   = {789--814},
  year    = {2000},
  doi     = {10.1016/S0005-1098(99)00214-9}
}

@article{lohmiller_slotine,
  author  = {Winfried Lohmiller and Jean-Jacques E. Slotine},
  title   = {On Contraction Analysis for Nonlinear Systems},
  journal = {Automatica},
  volume  = {34},
  number  = {6},
  pages   = {683--696},
  year    = {1998},
  doi     = {10.1016/S0005-1098(98)00019-3}
}

@article{rakovic_mrpi,
  author  = {Sa{\v{s}}a V. Rakovi{\'c} and Eric C. Kerrigan and
             Konstantinos I. Kouramas and David Q. Mayne},
  title   = {Invariant Approximations of the Minimal Robust Positively
             Invariant Set},
  journal = {IEEE Transactions on Automatic Control},
  volume  = {50},
  number  = {3},
  pages   = {406--410},
  year    = {2005},
  doi     = {10.1109/TAC.2005.843854}
}

@inproceedings{herbert_fastrack,
  author    = {Sylvia L. Herbert and Mo Chen and SooJean Han and
               Somil Bansal and Jaime F. Fisac and Claire J. Tomlin},
  title     = {{FaSTrack}: A Modular Framework for Fast and Guaranteed
               Safe Motion Planning},
  booktitle = {Proceedings of the IEEE Conference on Decision and Control
               (CDC)},
  pages     = {1517--1522},
  year      = {2017},
  doi       = {10.1109/CDC.2017.8263867}
}

\end{document}